\documentclass[a4paper,10pt]{amsart}

\usepackage{times}
\usepackage{amsmath,amssymb}
\usepackage{color}
\usepackage[normalem]{ulem}
\usepackage{a4wide}

\newtheorem{lemma}{Lemma}

\newtheorem{proposition}{Proposition}
\newtheorem{observation}{Observation}
\newtheorem{corollary}{Corollary}
\newtheorem{definition}{Definition}
\newtheorem{example}{Example}

\def\UpperBoundDimensionSeperator{13\,368}

\def\myemptyset{\varnothing}

\begin{document}

\title{Dimension of the Lisbon voting rules in the EU Council: a challenge and new world record}
\author{Sascha Kurz}
\address{Department of Mathematics, University of Bayreuth, 95440 Bayreuth, Germany.\\ Tel.: +49-921-557353, Fax: +49-921-557352, sascha.kurz@uni-bayreuth.de }
\author{Stefan Napel}
\address{Department of Economics, University of Bayreuth, and PCRC, University of Turku}

\begin{abstract}
The new voting system of the Council of the European Union cannot be represented as the intersection of six or fewer weighted games, i.e., 
its dimension is at least 7. This sets a new record for real-world voting bodies. A heuristic combination of different discrete optimization 
methods yields a representation as the intersection of {\UpperBoundDimensionSeperator} weighted games. Determination of the exact dimension 
is posed as a challenge to the community. The system's Boolean dimension is proven to be 3.
  
  \medskip
  
\end{abstract}

\maketitle

\section{Introduction}
\label{sec:Intro}
\noindent
Consider a group or committee whose members jointly decide whether to accept or reject a proposal (or, more generally, any system which outputs 1 if a minimal set of binary conditions are true and 0 otherwise).
The mapping of 
given configurations of approving members to a collective {\lq\lq}yes{\rq\rq} (1) or {\lq\lq}no{\rq\rq} (0) defines a so-called \emph{simple game}. It can often be described by a weighted voting rule:
each member~$i$ gets a non-negative weight $w_i$; a proposal is accepted iff the weight sum of its supporters meets a given quota $q$. The simple game is then known as a \emph{weighted game}.

Many real-word decision rules can be represented as weighted games, but not all. It is sometimes necessary to consider the intersection of multiple weighted games, or their union, in order to correctly delineate all 
acceptance and rejection configurations.
The minimal number of weighted
games whose intersection represents a given simple game is known as its \emph{dimension} \cite{MulticameralRepresentation}; the corresponding number in the disjunctive case is its \emph{co-dimension} \cite{NotionsOfDimension}.
The (co-)dimension of a rule which involves finitely many decision makers is finite, but can grow exponentially in the group size \cite[Thm.~1.7.5]{SimpleGamesBook}. It is NP-hard to determine the exact dimension of a given game \cite{DimensionNPhard}.

Taylor \cite{NoHighDimension} remarked in 1995 that he did not know of any real-world voting system of dimension $3$ or higher.
Amendment of the Canadian constitution \cite{CanadianConstitution} and the US federal legislative system \cite{MulticameralRepresentation} are classical examples of dimension 2.
More recently, systems of dimension $3$ have been adopted by the Legislative Council of Hong Kong \cite{LEGCO} and the Council of the European Union (EU~Council) 
under its Treaty of Nice rules \cite{DimensionEUNice}:
until late 2014, each EU member 
implicitly wielded a 3-dimensional vector-valued weight and 
proposals were accepted iff their supporters met a 3-dimensional quota. 
Real-world cases with dimension $4$ or more, however, have not been discovered yet (at least to our knowledge).
This suggests that determining the dimension of a given simple game might be a hard problem in theory but not in practice.

We establish that the situation is changed by the new voting rules of the EU Council, which were agreed to apply from Nov.\ 2014 on in the Treaty of Lisbon (with a transition period).
They involve a dual majority requirement where (i) at least 55\% of the EU member states need to support a motion and (ii) these supporters shall represent at least 65\% of the total EU population.
However, (iii) the {\lq\lq}no{\rq\rq}-votes of at least four EU member states are needed in order to block a proposal.
A coalition of the 25 smallest among the 28 EU members fails to meet provision (ii) but is still winning due to (iii).
We show that representing these rules as the union of one weighted game with the intersection of two more involves no redundancy, even for moderate changes of the current populations.
So the \emph{Boolean dimension} (see Def.~\ref{def_boolean_dimension}) of (i)--(iii) is 3, and robustly so.
Restricting representations to pure intersections or pure unions, however, increases the minimal number of weighted constituent games significantly.

We can prove that the dimension of the EU28's new voting rules is an integer between \textbf{$7$} and $\UpperBoundDimensionSeperator$; its co-dimension lies above $2000$.
This makes the EU28 a new record holder among real-world institutions.
The \emph{determination of the exact dimension} of voting rules in the EU~Council is an open computational challenge, which we here wish to present to a wider audience. It is related to the classical set covering problem in combinatorics and computer science.

The EU voting rules aside, the paper 
provides a general algorithmic approach for determining the dimension of simple games.
We combine combinatorial and algebraic techniques, exact and heuristic optimization methods in ways that are open to other applications and further refinements.
This contrasts with previously 
mostly tailor-made arguments for specific group decision rules. 



\section{Notation and definitions}
\noindent
We first introduce notation and some selected results on simple games;
\cite{SimpleGamesBook} is recommended for
a detailed treatment. 
Given a finite set $N=\{1,\dots,n\}$ of \emph{players}, a \emph{simple (voting) game} $v$ is a mapping $2^N\rightarrow\{0,1\}$ from the subsets of
$N$, called \emph{coalitions}, to $\{0,1\}$  (interpreted as a collective {\lq\lq}no{\rq\rq} and {\lq\lq}yes{\rq\rq})
which satisfies
$v(\myemptyset)=0$,
$v(N)=1$, and $v(S)\le v(T)$ for all $\myemptyset\subseteq S\subseteq T\subseteq N$.
Coalition $S\subseteq N$ is called \emph{winning} if $v(S)=1$ and \emph{losing} otherwise.
If $S$ is winning but all of its proper subsets are losing, then $S$ is called a \emph{minimal winning coalition}. Similarly, a losing coalition $T$ whose proper supersets are winning is called a \emph{maximal losing coalition}.
A simple game is more compactly 
characterized by its set $\mathcal{W}^m$ of minimal winning coalitions than by the corresponding set $\mathcal{W}$ of winning coalitions (or, equivalently, by  its set $\mathcal{L}^M$ of maximal losing coalitions rather than the set $\mathcal{L}$ of all losing coalitions).

Players of a simple game can often be ranked according to their `influence' or `desirability'.
Namely, if
$v(S\cup\{i\})\ge v(S\cup\{j\})$ for players $i,j\in N$ and all $S\subseteq N\backslash\{i,j\}$ then we write
$i\sqsupset j$ (or $j \sqsubset i$) and say that player~$i$ is \emph{at least as influential} as player~$j$. The case $i\sqsupset j$
and $j\sqsupset i$ is denoted as $i\,\square\,j$;  we then say that both players are \emph{equivalent}.
The $\square$-relation partitions the set of players into equivalence classes.
It is possible that
neither $i\sqsupset j$ nor $j\sqsupset i$ holds, i.e., players may be incomparable.
A simple game $v$ is called \emph{complete} if the binary relation $\sqsupset$ is complete, i.e., $i\sqsupset j$ or $j\sqsupset i$ for all $i,j\in N$.
Complete simple games form a proper subclass of simple games.

Given a complete simple game $v$, a minimal winning coalition $S$ is called \emph{shift-minimal winning} if
  $S\backslash\{i\}\cup\{j\}$ is losing for all $i\in S$ and all $j\in N\backslash S$ with $i\sqsupset j$
  but not $i\,\square\,j$, i.e., $S$ would become losing if any of its players~$i$ were replaced by a strictly less influential 
  player~$j$. Similarly, a maximal losing coalition $T$ is called \emph{shift-maximal losing} if
  $T\backslash\{i\}\cup\{j\}$ is winning for all $i\in S$ and 
  $j\in N\backslash S$ with $j\sqsupset i$
  but not $i\,\square\,j$.
A complete simple game is most compactly characterized by the partition of the players into equivalence classes and a description of either the shift-minimal winning or shift-maximal losing coalitions.

If there exist weights $w_i\in\mathbb{R}_{\ge 0}$
for all $i\in N$ and a quota $q\in\mathbb{R}_{>0}$ such that $v(S)=1$ iff $w(S):=\sum_{i\in S} w_i\ge q$
for all coalitions $S\subseteq N$ then we call the simple game $v$ \emph{weighted}.
Every weighted game is complete but the converse is false.
We call the vector $\left(q,w_1,\dots,w_n\right)$ a \emph{representation} of $v$
and write $v=\left[q;w_1,\dots,w_n\right]$.
If $v$ is weighted, there also exist representations such that all weights and the
quota are integers.
If $\sum_{i=1}^{n} w_i$ is minimal with respect to the integrality constraint, we speak of
a \emph{minimum sum integer representation} (see, e.g., \cite{MinimumSumIntegerRepresentation}).

If $v_1$, $v_2$ are weighted games with identical player set $N$ and respective sets of
winning coalitions $\mathcal{W}_1$ and $\mathcal{W}_2$ then the winning coalitions of $v_1\wedge v_2$ are given by $\mathcal{W}_1\cap\mathcal{W}_2$. The smallest number $k$ such that a simple game $v$ coincides with the intersection $v_1\wedge \ldots \wedge v_k$ of $k$ weighted games with identical player set is called the \emph{dimension} of $v$.
Similarly, the winning coalitions of $v_1\vee v_2$ are $\mathcal{W}_1\cup\mathcal{W}_2$, and the smallest number of weighted games whose union $v_1\vee \ldots \vee v_k$ coincides with a simple game $v$ is the \emph{co-dimension} of $v$.
Freixas and Puente have shown that there exists a complete simple game with dimension $k$ for every integer $k$ \cite{DimensionCSGwithMinimum}.
It is not known yet whether the dimension of a complete simple game is polynomially bounded in the number of its players or can grow exponentially (like for general simple games).

\begin{lemma} (cf.~\cite[Theorem 1.7.2]{SimpleGamesBook})
  \label{lemma_upper_bound_maximal_losing}
  The dimension of a simple game $v$ is bounded above by $\left|\mathcal{L}^M\right|$ and the co-dimension is bounded above by
  $\left|\mathcal{W}^m\right|$.
\end{lemma}

\begin{proof}
  For each coalition $S\in \mathcal{L}^M$ we set $q^S=1$, $w_i^S=0$ for all $i\in S$ and $w_i^S=1$ otherwise. Note that
  $\mathcal{L}^M\neq \myemptyset$ since $\myemptyset$ is a losing coalition. With this $w^S(S)=0<q^S$. However, for all $T\subseteq N$
  with $T\not\subseteq S$ we have $w(T)\ge 1=q^S$. Thus, we have $v=\bigwedge_{S\in\mathcal{L}^M} \left[q^S;w_1^S,\dots,w_n^S\right]$.
  Similarly, for each $S\in\mathcal{W}^n$ we set $\widetilde{q}^S=|S|$, $\widetilde{w}_i^S=1$ for all $i\in S$ and $w_i^S=0$ otherwise.
  Note that $\mathcal{W}^m\neq \myemptyset$ since $N$ is a winning coalition. With this $\widetilde{w}^S(S)=\widetilde{q}^S$. However, for all
  $T\subseteq N$ with $S\not\subseteq T$ we have $w(T)<\widetilde{q}^S$. Thus, we have
  $v=\bigvee_{S\in\mathcal{W}^m} \left[\widetilde{q}^S;\widetilde{w}_1^S,\dots,\widetilde{w}_n^S\right]$.
\end{proof}

Let $\Phi=\{u_1,\dots,u_k\}$ be a set of weighted games, interpreted as Boolean variables, and let $\varphi$ be a \emph{monotone Boolean formula} over $\Phi$, i.e., a well-formed formula of propositional logic over $\Phi$ which uses parentheses and the operators $\wedge$ and $\vee$ only.
The \emph{size} $|\varphi|$ of 
formula $\varphi$ is the number of
variable occurrences, i.e., the number of $\wedge$ and $\vee$ operators plus one. For instance, the size of
$u_1\vee(u_1\wedge u_2)$ is 3. 

\begin{definition}
  \label{def_boolean_dimension}
The \emph{Boolean dimension} of a simple game $v$ is the smallest integer $m$ such that there exist $k\le m$ weighted games $u_1, \ldots, u_k$ and a monotone Boolean formula $\varphi$ of size $|\varphi|=m$ satisfying $\varphi(u_1,\dots,u_k)=v$.
\end{definition}

Clearly, the Boolean dimension of $v$ is at most the minimum of $v$'s dimension and co-dimension.
Because combinations of $\wedge$ with $\vee$ have a size of at least 3, the Boolean dimension must exceed 2 whenever the dimension and co-dimension do.
The dimension can be exponential in the Boolean dimension of a simple game \cite[Thm.~4]{BooleanCombinations}; the Boolean dimension of a simple game can be exponential in the number of players \cite[Cor.~2]{BooleanCombinations}.

\section{Lisbon voting rules in EU Council}
\label{subsec_EU28}
We now formalize the provisions (i)--(iii) for decision making by the EU Council (see Sec.~\ref{sec:Intro}).
The membership requirement~(i)~-- approval of at least $16=\left\lceil 0.55\cdot 28\right\rceil$ member states~-- is easily reflected by the weighted game
$v_1=[16;1,\dots,1]$.
The population requirement~(ii) could be represented by using the official population counts as weights and 65\% of the total population as quota (see Table~\ref{table_pop_data_eu}).
\begin{table}[t]
  \begin{center}
    \setlength\tabcolsep{3pt}
    \begin{tabular}{rlrrp{8pt}rlrr}
      \hline
      \# & Member state & Population & $w_2$ && \# & Member state & Population & $w_2$\\
      \hline
       1 & Germany        & 80\,780\,000 & 4\,659\,052&& 16 & Bulgaria       &   7\,245\,677 &   417\,900\\
       2 & France         & 65\,856\,609 & 3\,798\,333&& 17 & Denmark        &   5\,627\,235 &   324\,556\\
       3 & United Kingdom & 64\,308\,261 & 3\,709\,031&& 18 & Finland        &   5\,451\,270 &   314\,406\\
       4 & Italy          & 60\,782\,668 & 3\,505\,689&& 19 & Slovakia       &   5\,415\,949 &   312\,369\\
       5 & Spain          & 46\,507\,760 & 2\,682\,373&& 20 & Ireland        &   4\,604\,029 &   265\,541\\
       6 & Poland         & 38\,495\,659 & 2\,220\,268&& 21 & Croatia        &   4\,246\,700 &   244\,932\\
       7 & Romania        & 19\,942\,642 & 1\,150\,208&& 22 & Lithuania      &   2\,943\,472 &   169\,767\\
       8 & Netherlands    & 16\,829\,289 & 970\,643&& 23 & Slovenia        &   2\,061\,085 &   118\,875\\
       9 & Belgium        & 11\,203\,992 & 646\,199&& 24 & Latvia          &   2\,001\,468 &   115\,436\\
      10 & Greece         & 10\,992\,589 & 634\,006&& 25 & Estonia         &   1\,315\,819 &    75\,890\\
      11 & Czech Republic & 10\,512\,419 & 606\,312&& 26 & Cyprus          &    858\,000 &    49\,486\\
      12 & Portugal       & 10\,427\,301 & 601\,403&& 27 & Luxembourg      &    549\,680 &    31\,703\\
      13 & Hungary        &  9\,879\,000 & 569\,780&& 28 & Malta           &    425\,384 &    24\,535\\
      14 & Sweden         &  9\,644\,864 & 556\,276&&    &                 &           &         \\
      15 & Austria        &  8\,507\,786 & 490\,693&&    & Total             & 507\,416\,607 & 2\,9265\,662        \\
      \hline
    \end{tabular}
    \caption{EU population 01.01.2014 (
     http://ec.europa.eu/eurostat); minimum sum integer weights of $v_2$} 
    \label{table_pop_data_eu}
  \end{center}
\end{table}
Its computationally more convenient minimum sum integer representation is given by
$v_2=[q; \mathbf{w}_2]$ with $q=19\,022\,681$ and the weights indicated in the $w_2$-columns of Table~\ref{table_pop_data_eu}.\footnote{We remark that \emph{rounding} 
populations to, say, thousands is common in applied work because this simplifies computations, e.g., of the voting power distribution in the EU Council. Rounding, however, 
leads to a different set of winning coalitions, i.e., is analyzing `wrong' rules.} 
%
%
The additional minimal blocking requirement (iii) can be described as $v_3=[25;1,\dots,1]$, since $28-4+1=25$ member states suffice to pass a proposal.
The Lisbon voting rule of the EU Council is then formally characterized as
$v_{\text{EU28}}=(v_1\wedge v_2)\vee v_3$ or $v_{\text{EU28}}=v_1\wedge (v_2\vee v_3)$.

The $268\,435\,456$ coalitions
of $v_{\text{EU28}}$ are partitioned into $30\,340\,718$ winning and $238\,094\,738$ losing coalitions. Of these, $8\,248\,125$ are minimal winning and $7\,179\,233$ maximal losing. So the dimension of $v_{\text{EU28}}$ must be below $7.18$ millions.

The influence partition of the Boolean combination of weighted games generally corresponds to the coarsest common refinement of the respective partitions in the constituent games. Here, 
there is only a single equivalence class of players in $v_1$ and $v_3$, respectively, while $v_2$ has $28$ equivalence classes (all minimum sum weights differ by at least 2). So each player forms its own equivalence class in $v_{\text{EU28}}$.
There are only $60\,607$ shift-minimal winning and $60\,691$ shift-maximal losing coalitions
in $v_{\text{EU28}}$.\footnote{For example, every $16$-member winning coalition is minimal but few are also shift-minimal.}

\section{Weightedness and bounding strategy}
\label{sec_weightedness}

\noindent
Determining whether a given simple game is weighted or not will be crucial for our analysis of $v_{\text{EU28}}$.
Answers can be given by combinatorial, algebraic or geometric methods (see \cite[Ch.~2]{SimpleGamesBook}). We will draw on the first two.

Combinatorial techniques usually invoke so-called `trades'.
A \emph{trading transform} for a simple game $v$ is a collection of coalitions $J=\langle S_1,\ldots, S_j;T_1,\ldots,T_j\rangle$ such that $\left|\{h\colon i\in S_h\}\right|=\left|\{h\colon i\in T_h\}\right|$ for all $i\in N$.
An \emph{$m$-trade} for $v$ is a trading transform with $j\le m$ such that all $S_h$ are winning and all $T_h$ are losing coalitions.
Existence of, say, a 2-trade $\langle S_1,S_2;T_1,T_2\rangle$ implies that the game cannot be weighted: $w(S_1), w(S_2)\ge q$ and $w(T_1), w(T_2)<q$ would contradict $w(S_1)+w(S_2)=w(T_1)+w(T_2)$.
The simple game $v$ is called \emph{$m$-trade robust} if no $m$-trade exists for it.
Taylor and Zwicker have shown that a simple game is weighted iff it is $m=2^{2^n}$-trade robust (see, e.g., \cite[Thm.~2.4.2]{SimpleGamesBook}). Sharper bounds for $m$ have been provided by \cite{gvozdeva2011weighted}, but the lower one is still linear and the upper exponential in $n$.

\begin{example}
  \label{ex_not_weighted}
  Consider the complete simple game $v$ with $N=\{1,2,3,4,5,$ $6\}$ and $$\mathcal{L}^M=\big\{\{1,3,5\}, \{1,3,6\},  \{1,4,5\}, \{1,4,6\}, \{2,3,5\}, \{2,3,6\}, \{2,4,5\}, \{2,4,6\}\big\}.$$
All coalitions in $\mathcal{L}^M$ are also shift-maximal losing,
but only coalitions $\{1,2\}$, $\{1,3,4\}$, $\{2,3,4\}$
  and $\{3,4,5,6\}$ of $$\mathcal{W}^m
=\big\{\{1,2\}, \{1,3,4\}, \{2,3,4\}, \{3,4,5,6\}, \{1,3,5,6\},  \{1,4,5,6\},  \{2,3,5,6\}, \{2,4,5,6\}\big\}$$
 are also shift-minimal winning. Since 
  $$
    \left\langle \{1,2\},  \{3,4,5,6\}; \{1,3,5\}, \{2,4,6\}\right\rangle
  $$
is a $2$-trade, $v$ is not weighted.\footnote{The example is the smallest possible: all complete simple games with $n\le 5$ are weighted.}
\end{example}

Algebraic methods exploit that a simple game $v$ is weighted iff the inequality system
$
  \sum_{i\in S} w_i \ge q\,\, \forall S\in\mathcal{W}^m\!, \,
  \sum_{i\in T} w_i \le  q-1\,\,\forall T\in\mathcal{L}^M\!,\,
  w_i\in\mathbb{R}_{\ge 0}\,\forall i\in N,\text{ and }
  q \in\mathbb{R}_{\ge 1}
$
admits a solution. 
Linear programming (LP) techniques can be applied.
In case that no solution exists, the dual multipliers provide a certificate of non-weightedness.
A suitable subset of the constraints~-- those for the minimal winning and some maximal losing coalitions, say~-- often suffice to conclude infeasibility and thus non-weightedness.

For a \emph{complete} simple game $v$ with sets $\mathcal{W}^{sm}$ and $\mathcal{L}^{sM}$ of shift-minimal winning and shift-maximal losing coalitions, the linear inequality system can further be simplified. Namely, $v$ is weighted iff
\begin{equation}
\begin{array}{l}\textstyle
  \sum\limits_{i\in S} w_i \ge q\,\, \forall S\in\mathcal{W}^{sm}\!, \ \
  \sum\limits_{i\in T} w_i \le  q-1\,\,\forall T\in\mathcal{L}^{sM}\!,\medskip\\
  w_i\ge w_j\in\mathbb{R}_{\ge 0}\,\forall i,j\in N\text{ with }i\sqsupset j, \ \
  w_i\in\mathbb{R}_{\ge 0}\,\forall i\in N \text{ and } q \in\mathbb{R}_{\ge 1}
\end{array}\label{LP_CSG}
\end{equation}
admits a solution.
Note that non-weightedness of $v$ says no more about $v$'s dimension than that it exceeds 1.

One might hope that it is possible to construct a representation of a complete simple game $v$ as the intersection of $\left|\mathcal{L}^{sM}\right|$ weighted games as follows:
look at one coalition $T_l\in \mathcal{L}^{sM}$ at a time; find a weighted game $v_l$ such that (a) $v_l(T_l)=0$ and (b) $v_l(S)=1$ for every $S\in \mathcal{W}^{sm}$ by ignoring all constraints $\sum_{i\in T'} w_i \le  q-1$ in system~(\ref{LP_CSG}) for  $T'\in\mathcal{L}^{sM}\setminus T_l$; finally obtain $v_1\wedge \ldots \wedge v_{|\mathcal{L}^{sM}|}$ as a representation of $v$.
Unfortunately, this does not work in general. For instance, we can infer from infeasibility of
$w_1+w_2\ge q$, $w_3+w_4+w_5+w_6\ge q$, $w_1+w_3+w_5\le q-1$, $w_1=w_2$, $w_3=w_4$ and $w_5=w_6$
that there exists no weighted game $v_1$ which respects the ordering condition $w_i\ge w_j \Longleftrightarrow i\sqsupset j $ and in which $T_1=\{1,3,5\}\in \mathcal{L}^{sM}$ is losing and (at least) $\{1,2\}$ and $\{3,4,5,6\}$ are winning (see Example~\ref{ex_not_weighted}).
Counter-examples exist also when no two players are equivalent.
The basic idea of this heuristic construction is still useful, and will be applied 
in order to provide an \emph{upper bound} on $v_{\text{EU28}}$'s dimension.
In order to establish a \emph{lower bound},
we will use 
\begin{observation}
Given a simple game $v$ with winning coalitions $\mathcal{W}$ and losing coalitions $\mathcal{L}$,
let $\mathcal{L}'=\left\{T_1,\dots,T_k\right\}
\subseteq \mathcal{L}$ be a set of losing coalitions with the following `pairwise incompatibility property': for each pair $\{T_i,T_j\}$ with $T_i\neq T_j\in \mathcal{L}'$ 
there exists \emph{no} weighted game in which all coalitions in $\mathcal{W}$ are winning while $T_i$ and $T_j$ are both losing.
Then if $v=\bigwedge_{1\le l\le m} v_l$ is the intersection of $m$ weighted games, we must have $m\ge k$, i.e., $v$'s dimension is at least $k$.
\end{observation}
The observation generalizes the construction used in \cite{DimensionEUNice}.
A quick way to establish that there is no weighted game with $T_i$ and $T_j$ losing and all $S\in \mathcal{W}$ winning is to find a 2-trade $\left\langle S_1, S_2; T_i, T_j\right\rangle$ for some $S_1, S_2\in \mathcal{W}$.
Not finding a 2-trade does not guarantee that such weighted game exists; and checking for 3-trades, 4-trades, etc.\ gets computationally demanding.
However, in order to provide a lower bound $k$ for $v_{\text{EU28}}$'s dimension, it suffices to provide \emph{any} set $\mathcal{L}'$ of $k$ pairwise incompatible losing coalitions.
So one can focus on sets in which 2-trades are easily obtained for all ${k \choose 2}$ pairs, and improve the resulting bound 
by extending $\mathcal{L}'$ if needed.

We remark that it is possible to formulate the \emph{exact determination} of the
dimension of a simple game 
as a discrete optimization problem.
To this end let $\mathcal{C}$ collect all subsets $\mathcal{S}\subseteq \mathcal{L}^M$ 
with the property that there exists a weighted game where all elements of $\mathcal{W}$ are winning and all elements of
$\mathcal{S}$ are losing.
In particular, all singleton subsets of $\mathcal{L}^M$ are contained in $\mathcal{C}$ (cf.\ proof of Lemma~\ref{lemma_upper_bound_maximal_losing}); so is, e.g.,  $\{\{1,3,5\},\{1,3,6\},\{1,4,5\},\{1,4,6\}\}$ in Example~\ref{ex_not_weighted}, but not  $\{\{1,3,5\},\{2,4,6\}\}$.

Having constructed $\mathcal{C}$, the dimension of $v$ can be determined by finding a \emph{minimal covering} of $\mathcal{L}^M$, using the elements of $\mathcal{C}$.
Specifically, $v$'s dimension is the optimal value of $\min \sum_{\mathcal{S}\in \mathcal{C}}
x_{\mathcal{S}}$ subject to the constraints $\sum_{\mathcal{S}\in\mathcal{C}:\,T\in\mathcal{S}} x_{\mathcal{S}}\ge 1$
for all $T\in\mathcal{L}^M$ and $x_{\mathcal{S}}\in\{0,1\}$ for all $\mathcal{S}\in\mathcal{C}$. However, this set covering formulation is, in general, computationally intractable. For $v_{\text{EU28}}$, already the construction of $\mathcal{C}$ is out of reach because $\mathcal{L}^M$ has more than $2^{7.1\cdot 10^6}$ subsets.
We hence have to contend ourselves with lower and upper bounds which may be brought to identity at some point in the future.

\section{Bounds for $v_{\text{EU28}}$'s dimension}
\label{sec_bounds}

Since $v_{\text{EU28}}$ has so many maximal losing coalitions we have focused our search for a suitable set $\mathcal{L}'$ of pairwise incompatible losing coalitions on the subset  $\mathcal{L}_{23,24}\subset \mathcal{L}$ of losing coalitions with $23$ or $24$ members.
They fail the 65\% population and 25 member thresholds. For each
pair of these 4\,533 coalitions we have performed a greedy search for a $2$-trade. Specifically, let two such losing coalitions $T_i\neq T_j\in \mathcal{L}_{23,24}$ be given, set $I=T_i\cap T_j$, and then extend $I$ to a winning coalition $S_1$ with $25$ members
by choosing the least populous elements of $\left(T_i\cup T_j\right)\backslash I$.
Coalition $S_2$ is then defined by $(\left(T_i\cup T_j\right)\backslash S_1)\cup I$. If $S_2$ is winning, we have found a $2$-trade, i.e., pair $\left\{T_i,T_j\right\}$ satisfies the incompatibility criterion.
Marking this occurrence as an edge in a graph $\mathcal{G}$ with vertex set $\mathcal{L}_{23,24}$, we can perform a clique search on $\mathcal{G}$. It turns out that
$\mathcal{G}$ contains 
$12\,226\,400$
cliques of size $6$ but no larger clique.
One of the 6-cliques corresponds to $\mathcal{L}'=$

\begin{eqnarray*}
\Big\{ && \!\!\!\!\{3,4,5,6,7,8,12,13,14,15,16,17,18,19,20,21,22,23,24,25,26,27,28\},\\[-2mm]
  && \!\!\!\!\{2,4,5,6,7,9,10,11,13,14,15,17,18,19,20,21,22,23,24,25,26,27,28\},\\[-1mm]
  && \!\!\!\!\{2,3,5,6,8,9,10,11,12,15,16,17,18,19,20,21,22,23,24,25,26,27,28\},\\[-1mm]
  && \!\!\!\!\{2,3,4,7,8,9,10,11,12,13,14,16,18,19,20,21,22,23,24,25,26,27,28\},\\[-1mm]
  && \!\!\!\!\{1,4,5,7,8,9,10,11,12,13,14,15,16,17,20,21,22,23,24,25,26,27,28\},\\[-2mm]
  && \!\!\!\!\{1,3,6,7,8,9,10,11,12,13,14,15,16,17,18,19,22,23,24,25,26,27,28\}\,\Big\}.\footnote{}
\end{eqnarray*}\footnotetext{Just to give an example,
$\big\langle\{4,\dots,28\},
\{2,\dots,7,13,\dots,15,17,\dots,28\}; 
\{3,\dots,8,11,13,\dots,28\}, 
\{2,4,\dots,7,9, 10,,11,$ $13,\dots,15,17,\dots,28\}\big\rangle$ is a $2$-trade for the first two losing coalitions. Incorrect losing coalitions and wrong 
$\Vert \cdot \Vert_1$-distance were reported in the published version: S. Kurz and S. Napel (2016). Dimension of the Lisbon
voting rules in the EU Council: a challenge and new world record. Optimization Letters, 10(6), 1245–1256, doi:10.1007/s11590-015-0917-0, due to a labeling
inconsistency. The authors thank J\"org Aldag and Werner Kirsch for pointing out this error.}

This 6-clique is actually the most robust one regarding changes of the relative population distribution in the EU: it is not upset by moves between
states, births, or deaths as long as the new relative population vector $pop'$ and the old one, $pop$, based on Table~\ref{table_pop_data_eu}, have a $\Vert \cdot \Vert_1$-distance less than 
0.0068. This distance could accommodate arbitrary moves of up to 1.7~million EU citizens.
The robustness is noteworthy because high numbers in the minimum sum representation of $v_2$ indicate that $v_{\text{EU28}}$ is rather sensitive to population changes. 

The above set $\mathcal{L}'$ can be extended, without affecting robustness, by adding the maximal losing coalition $\{1, \ldots, 15\}$ of the 15 largest member states, which was excluded by the initial focus on $\mathcal{L}_{23,24}$. This establishes:
\begin{proposition}\label{prop_robust7}
Let $v$ be the simple game arising from $v_{\text{EU28}}$ by replacing the underlying relative population vector $pop$ by the relative population vector $pop'$. 
If $\Vert pop-pop'\Vert_1\le 0.68\%$
then $v$ has \emph{dimension at least $7$}.
\end{proposition}

An alternative for establishing a lower bound $d$ for $v_{\text{EU28}}$'s dimension is to replace the graph-theoretic search for 2-trades 
by a straightforward integer linear program (ILP) such as\footnote{For the general ILP modeling of weighted games we refer to \cite{InversePowerIndexProblem}.}
\begin{eqnarray*}
  \max \Delta\quad \,\,\text{s.t.}\quad \,\,
  \sum_{i=1}^{28} l_i^j \le 24\,\,\forall 1\!\le\! j\!\le\! d,
  \sum_{i=1}^{28} pop_i\cdot l_i^j \le 0.65-\Delta\,\,\forall 1\!\le\! j\!\le\! d\\
   \sum_{i=1}^{28} w_{i}^{j,h,1} \ge 25\,\,\forall 1\!\le\! j\!<\!h\le\! d,
  \sum_{i=1}^{28} pop_i\cdot w_{i}^{j,h,2} \ge 0.65+\Delta\,\,\forall 1\!\le\! j\!<\!h\!\le d\\
l_i^j\!+\!l_i^h=w_{i}^{j,h,1}\!+\!w_{i}^{j,h,2} \,\,\forall 1\!\le\! i\!\le\! 28, 1\!\le\! j\!<\!h\!\le\! d,
  l_i^j\!\in\!\{0,1\}\,\,\forall 1\!\le\! i\!\le\! 28, 1\!\le\! j\!\le\! d\\
\sum_{i=1}^{28} w_{i}^{j,h,2} \!\ge\! 16\,\forall 1\!\le\! j\!<\!h\le\! d, w_{i}^{j,h,k} \!\in\!\{0,1\}\,\,\forall 1\!\le\! i\!\le\! 28, 1\!\le\! j\!<\!h\!\le\! d, k\!\in\!\{1,2\}.
\end{eqnarray*}
This turned out to be impractical for $d>6$ but has yielded a simple, robust certificate for $d=3$, which will be useful for obtaining Corollary~\ref{cor_codimension} below:


\begin{proposition}\label{prop_robust3}
Let $v$ be the simple game arising from $v_{\text{EU28}}$ by replacing the underlying population vector $pop$ by the relative population vector $pop'$. If $\Vert pop-pop'\Vert_1\le 2.19\%$ then $v$ has dimension at least $3$.
\end{proposition}
\begin{proof}
Three losing coalitions whose pairs can be completed to a $2$-trade are: 
\begin{eqnarray*}
&& \{1,4,5,7,8,9,11,\dots,14,16,\dots,26,28\},\,\,
\{3,\dots,6,8,9,10,12,14,\dots,24,26,27\},\text{ and}\\
&&\{2,4,\dots,8,10,11,13,15,17,\dots,20,22,\dots,25,27,28\}.
\end{eqnarray*}
\end{proof}
%
%

\medskip

In order to bring down the baseline upper bound of $\left|\mathcal{L}^{M}\right|\approx 7.18$~mio.\ for $v_{\text{EU28}}$'s dimension (Lemma~\ref{lemma_upper_bound_maximal_losing}), we
draw on LP formulation (\ref{LP_CSG}) and the indicated idea to check for each $T_l \in \mathcal{L}^{sM}$ whether inequality system~(\ref{LP_CSG}) with $\mathcal{L}^{sM}$ replaced by $\{T_l\}$ has a feasible solution.
This yields weighted games for $57\,869$ out of $\left|\mathcal{L}^{sM}\right|=60\,691$ coalitions.
The remaining $2\,822$ \emph{stubborn} shift-maximal losing coalitions correspond to exactly
$17\,003$ maximal losing coalitions, which are not yet covered by the identified weighted games.
We could apply the construction in the proof of Lemma~\ref{lemma_upper_bound_maximal_losing} to these and would obtain an upper bound of $74\,872$.

This, however, is easily improved by the following procedure:
(I) try to greedily cover 
many shift-maximal losing coalitions 
with a few selected weighted games; (II) find a weighted game $v_j$ for each still uncovered and non-stubborn $T_j \in \mathcal{L}^{sM}$; (III) 
deal with the maximal losing coalitions related to all stubborn $T_k$.
We utilized the following ILP in order to iteratively find helpful games in step (I)
\begin{eqnarray*}
  \max \sum_{T\in\mathcal{L}''} x_T \text{ s.t. }%
  x_T\in\{0,1\} \forall T\in\mathcal{L}'',
  w_i\ge w_{i+1}\,\, \forall 1\le i\le 27,
  \sum_{i=1}^{28} w_i\le M ,
  \\
    \sum_{i\in S} w_i\ge q \,\forall S\in\mathcal{W}^{sm},
  \sum_{i\in T} w_i\le q\!-\!1\! +\!(1\!-\!x_T)M,
  w_i,q\in \mathbb{N}\,\forall 1\!\le\! i\!\le\! 28.
  \end{eqnarray*}
This ILP  exploits that $1\sqsupset  \ldots \sqsupset 28$ in $v_{\text{EU28}}$, the 
constant $M$ is chosen so as to give integer weights with suitable magnitude (e.g., thousands), and $\mathcal{L}''$ is the part of $\mathcal{L}^{sM}$ which is still uncovered 
or a subset thereof.
It is possible, for instance, to cover $34\,323$ shift-maximal losing coalitions in step (I) with just $10$ weighted games. Adding more weighted games to these,
the \emph{lowest upper bound} which we have obtained so far is $\UpperBoundDimensionSeperator$. The 
games and a checking tool can be obtained from the authors. 

All of these 
considerations can easily be translated to the \emph{co-dimension}.
There, we have to consider unions of weighted games, where all coalitions in $\mathcal{L}^M$ are losing and the winning coalitions in $\mathcal{W}^m$ end up being covered by a suitable selection of constituent games.
We skip the 
details for space reasons.

\begin{proposition}\label{prop_codimension}
  Let $v$ be the simple game arising from $v_{\text{EU28}}$ by replacing the underlying relative  population vector $pop$ by
  the relative population vector $pop'$. If $\Vert pop-pop'\Vert_1\le 5\%$ then $v$ has \emph{co-dimension at least $7$}.
\end{proposition}
\begin{proof}
Seven winning coalitions whose pairs can be completed to a $2$-trade are:\\
  $\{2,\dots,5,7, 8,9,11,\dots,15,17,\dots,20\}$,
  $\{1,2,3,6,8,\dots,15,17,18,19,25\}$,\\
  $\{1,3,5,\dots,16,19,20\}$
  $\{1,2,5,\dots,17,22\}$,
  $\{1,2,4,5,7,9,\dots,15,19,23,24,26\}$,\\
  $\{1,2,4,6,8,\dots,16,18,20,21\}$,
  and
  $\{1,2,3,5,7,10,\dots,15,18,20,21,22,28\}$.
\end{proof}
The combination of Propositions~\ref{prop_robust3}
and \ref{prop_codimension} yields:
\begin{corollary}\label{cor_codimension}
  Let $v$ be the simple game arising from $v_{\text{EU28}}$ by replacing the underlying relative  population vector $pop$ by
  the relative population vector $pop'$. If $\Vert pop-pop'\Vert_1\le 2.19\%$ then $v$ has \emph{Boolean dimension exactly $3$}.
\end{corollary}

We remark that is not too hard to determine $2\,000$ winning coalitions such that each pair can be completed to a $2$-trade. So the co-dimension of $v_{\text{EU28}}$ with populations exactly as in Table~\ref{table_pop_data_eu} is at least $2\,000$.

\section{Concluding remarks}
\label{sec_conclusion}

Simple game $v_3$ rules out that three of the EU's ``Big Four'' (see Table~\ref{table_pop_data_eu}) can cast a veto in the Council. This has 
very minor consequences for the mapping of different voting configurations to a collective ``yes'' or ``no'':
the disjunction with
$v_3$ adds a mere $10$ to the $30\,340\,708$ coalitions which are already winning in $v_1\wedge v_2$.
Prima facie, provision~(iii) should therefore have only symbolic influence on the distribution of voting power in the Council.\footnote{In order to check this intuition, we have computed the difference $\left\Vert \mathcal{P}(v_{\text{EU28}})-\mathcal{P}(v_1\wedge v_2)\right\Vert_1$ for four different power measures $\mathcal{P}$ (cf.\ \cite{alonso2014least}):
it is only around $7\cdot 10^{-7}$ for the least square nucleolus and $9\cdot 10^{-7}$ for the normalized Banzhaf index, but $0.00537$ for the Shapley-Shubik index and $0.167$ for the nucleolus.}
Quite surprisingly, however, provision~(iii) has tremendous effect on the conjunctive dimensionality of the rules.
Namely, the EU Council sets a new world record, among the political institutions that we know of: 
the dimension of its decision rule is at least 7. 

The link to classical set covering problems in optimization which we have identified and partly exploited in Sections~\ref{sec_weightedness} and \ref{sec_bounds} implies that there exist algorithms which should~-- at least in theory~-- terminate with an answer to the simple question: what is the dimension of $v_{\text{EU28}}$?
In practice, heuristic methods which establish and improve bounds are needed.
The suggested mix of combinatorial and algebraic techniques, integer linear programming and graph-theoretic methods has rather general applicability. It also lends itself to robustness considerations, which we hope will become more popular in the literature.
(A potentially negative referendum on EU membership in the UK and a consequent exit, for instance, would leave our lower bounds intact.)

The drawback of our relatively general approach is that the resultant upper bound of $\UpperBoundDimensionSeperator$ 
is still pretty high;
the record lower bound of 7 may not be the final word either. Alternative approaches, which might use unexploited specifics of $v_{\text{EU28}}$, 
will potentially lead to much sharper boundaries in the future.

The certification of better dimension bounds is a problem which we would here like to advertise to the optimization community.
The application of meta-heuristics, such as simulated annealing and genetic algorithms, or column generation techniques could be promising.
The ultimate challenge is, of course, to determine the \emph{exact dimension} of the group decision rule in the EU Council.



\end{document}